\documentclass[11pt]{article}

\usepackage{tikz}
\usepackage{subfigure}
\usepackage[english]{babel}

\usepackage[center]{caption2}
\usepackage{amsfonts,amssymb,amsmath,latexsym,amsthm}
\usepackage{multirow}

\def\ex{\mbox{ex}}
\def\spex{\mbox{spex}}

\newtheorem{thm}{Theorem}[section]

\newtheorem{lem}[thm]{Lemma}
\newtheorem{prop}[thm]{Proposition}

\newtheorem{claim}{Claim}

\newtheorem{case}{Case}

\begin{document}

\title{Spectral extrema of graphs with bounded clique number and matching number
%\thanks{The work was supported by NNSF of China (No. 12071453) and Anhui Initiative in Quantum Information Technologies (AHY150200) and National Key Research and Development Project (SQ2020YFA070080).}
}
\author{Hongyu Wang$^a$, \quad Xinmin Hou$^{a,b}$, \quad Yue Ma$^a$\\
\small $^{a}$ School of Mathematical Sciences\\
\small University of Science and Technology of China, Hefei, Anhui 230026, China.\\
\small  $^{b}$ CAS Key Laboratory of Wu Wen-Tsun Mathematics\\
\small University of Science and Technology of China, Hefei, Anhui 230026, China.\\
\small Hefei, Anhui 230026, China.
}

\date{}

\maketitle

\begin{abstract}
For a set of graphs $\mathcal{F}$, 
let $\ex(n,\mathcal{F})$ and $\spex(n,\mathcal{F})$ denote the maximum number of edges and the maximum spectral radius of an $n$-vertex $\mathcal{F}$-free graph, respectively.
%and let $\spex(n,\mathcal{F})$ denote the maximum spectral radius of an $n$-vertex $\mathcal{F}$-free graph.
Nikiforov ({\em LAA}, 2007) gave the spectral version of the Tur\'an Theorem by showing that $\spex(n, K_{k+1})=\lambda (T_{k}(n))$, where $T_k(n)$ is the $k$-partite Tur\'an graph on $n$ vertices. In the same year, Feng, Yu and Zhang ({\em LAA}) determined the exact value of $\spex(n, M_{s+1})$, where $M_{s+1}$ is a matching with $s+1$ edges. 
Recently, Alon and Frankl~(arXiv2210.15076) gave the exact value of $\ex(n,\{K_{k+1},M_{s+1}\})$. 
In this article, we give the spectral version of the result of Alon and Frankl by determining the exact value of $\spex(n,\{K_{k+1},M_{s+1}\})$ when $n$ is large.
\end{abstract}

\section{Introduction}
In this paper, we consider only simple and finte graphs. Let $G = (V, E)$ be a graph with vertex set $V = V (G)$ and edge set $ E = E(G)$. We write $|G|$ for $|E(G)|$ through this paper.   Let $A(G)$ be the {\em adjacency matrix} of $G$ and let $\lambda(G)$ be the largest eigenvalue of $A(G)$, and call
it  the {\em spectral radius} of $G$.

Let $\mathcal{F}$ be a family of graphs, we say graph $G$ is {\em $\mathcal{F}$-free} if $G$ does not contain any graph in $\mathcal{F}$ as a subgraph. As the classical Tur\'an type problem determines the maximum number of edges of an $n$-vertex $\mathcal{F}$-free graph, called {\em Tur\'an number} and denoted by $\ex(n,\mathcal{F})$. Brualdi-Solheid-Tur\'an type problems consider the maximum spectral radius of an $n$-vertex $\mathcal{F}$-free graph, denoted by $\spex(n,\mathcal{F})$, i.e. $$\spex(n,\mathcal{F})=\max\{\lambda(G) : G \text{ is an $n$-vertex $\mathcal{F}$-free graph}\}.$$
In the recent ten years, there are fruitful results of Brualdi-Solheid-Tur\'an type problems, {for example, in~\cite{Babai,odd wheels,Ck,T4,subtrees,NWK23,surveyN,Pk,survey,Kst-minort,Kst-minorz}. }

 Let $K_n$  and $\overline{K_n}$ denote the complete graph and the empty graph on $n$ vertices, respectively. For any graph $G$ and $U\subset V(G)$, write $G-U=G[V(G)\backslash U]$. 
 Let $K_{n_1,\cdots,n_k}$ denote the complete $k$-partite graph with partition sets of sizes $n_1,\ldots,n_k$.
 A Tur\'an graph $T_{k}(n)$ is the complete $k$-partite graph on $n$ vertices whose partition sets have sizes as equal as possible.
Define  $G_k(n,s)=T_{k-1}(s)\vee \overline{K_{n-s}}$, the join of the Tur\'an graph $T_{k-1}(s)$ and  empty graph $\overline{K_{n-s}}$. Clearly,  $G_k(n,s)$ is a complete $k$-partite graph on $n$ vertices with one partition set of size $n-s$ and the others having sizes as equal as possible. 
% Given a graph $H$  with $\chi(H)\ge 3$, define $\mathcal{H}(H)$ to be the family of graphs obtained by deleting a color class from $H$.
% Define $D_H(n, s)=D\vee \overline{K_{n-s}}$, where $D$ is a copy of extremal graph of $\mathcal{H}(H)$ on $s$ vertices.
 Write $M_k$ for a matching consisting of $k$ edges.

A fundamental theorem  (Tur\'an Theorem) due to Tur\'an~\cite{Tr}  gives $\ex(n;K_{k+1}) = |E(T_{k}(n))|$ for $n > k + 1 > 3$. In 2007, Nikiforov~\cite{Niki} gave a spectral version of the Tur\'an Theorem by showing that $\lambda(G)\leq \lambda (T_{k}(n))$ for every $n$-vertex $K_{k+1}$-free graph $G$, with equality if and only if $G\cong T_{k}(n)$. When considering the bounded matching number, Feng, Yu and Zhang~\cite{zxd} proved that 
$$\spex(n, M_{s+1})=\begin{cases}
	\lambda(K_n), \mbox{ if $n=2 s$ or $2 s+1$}; \\
	\lambda(K_{2 s+1} \cup \overline{K_{n-2 s-1}}), \mbox{ if $2 s+2 \leqslant n<3 s+2$};\\ 
	\lambda(K_s \vee \overline{K_{n-s}}) \mbox{ or } \lambda(K_{2 s+1} \cup \overline{K_{n-2 s-1}}), \mbox{ if $n=3 s+2$};\\
	\lambda(K_s \vee \overline{K_{n-s}}), \mbox{ if $n>3 s+2$}.
\end{cases}
$$
Recently, Ni, Wang and Kang~\cite{NWK23} extended the above result by determining the exact value of $\spex(n, kK_{r+1})$ for $k\ge 2$, $r\ge 2$, and sufficiently large $n$. 
%for all graphs of order $n$ with matching number $s$, the graphs with maximal spectral radius are $K_n$ if $n=2 s$ or $2 s+1 ; K_{2 s+1} \cup \overline{K_{n-2 s-1}}$ if $2 s+2 \leqslant n<3 s+2 ; K_s \bigvee \overline{K_{n-s}}$ or $K_{2 s+1} \cup \overline{K_{n-2 s-1}}$ if $n=3 s+2 ; K_s \bigvee \overline{K_{n-s}}$ if $n>3 s+2$.

%Let $G_k(n,s)$ denote the complete $k$-partite graph on $n$ vertices consisting of one vertex class of size $n-s$ and the rest $k-1$ vertex classes of size as equal as possible.

Another fundamental result in graph theory is the
Erd\H{o}s-Gallai Theorem~\cite{EG59}, showing that $$\ex(n, M_{s+1})=\max\left\{\left|E(G_{s+1}(n, s))\right|, {2s+1\choose s+1}\right\}.$$ 
Recently, Alon and Frankl~\cite{AF} combined the forbidden graphs of Tur\'an Theorem and Erd\H{o}s-Gallai Theorem by showing that
\begin{thm}[\cite{AF}]\label{AF}
For $n\ge 2s+1$ and $k\ge 2$, 
$$\ex(n,\{K_{k+1},M_{s+1}\})=\max\{|T_k(2s+1)|,|G_k(n,s)|\}\mbox{.}$$
\end{thm}
{Observe that when $n$ is sufficiently large,  $$\ex(n,\{K_{k+1},M_{s+1}\})=\max\{|T_k(2s+1)|,|G_k(n,s)|\}=|G_k(n,s)|.$$} In this note, we consider the Brualdi-Solheid-Tur\'an type problem of Theorom~\ref{AF} when $n$ is sufficiently large. Here is our main theorem.
\begin{thm}\label{main}
For $n\ge 4s^2+9s$ and $k\ge 2$, 
$$\spex(n,\{K_{k+1},M_{s+1}\})=\lambda(G_k(n,s))\mbox{.}$$
%{\color{red} Is the extremal graph unique?}
\end{thm}

The rest of the note is arranged as follows. We give some preliminaries and lemmas. The proof of Theorem~\ref{main} will be given in Section 2. We give some discussion in the last section.

\section{Preliminaries and lemmas}
The Tutte-Berge Theorem~\cite{TB} (also see the Edmonds-Gallai Theorem~\cite{EG}) is very useful when we cope with the problem related to matching number.
\begin{lem}[\cite{TB},\cite{EG}]\label{Tutte-Berge}
	A graph $G$ is $M_{s+1}$-free if and only if there is a set $B\subset V(G)$ such that  all the components $G_1,\ldots, G_m$ of $G-B$ are odd (i.e. $|V(G_i)|\equiv1\pmod2$ for $i\in[m]$), and 
	$$|B|+\sum_{i=1}^{m}\frac{|V(G_i)|-1}{2}=s\mbox{.}$$
%For a graph $G$, the matching number of $G$ is $s$ if and only if there is a set of vertices $B$, $|B|=b$ so that all of the connected components $A_1,\cdots,A_m$ of $G-B$ are odd. Let $|A_i|=a_i, i=1\cdots m$, we have that
%$$b+\sum_{i=1}^{m}\frac{a_i-1}{2}=s$$ and $$b+\sum_{i=1}^m a_i=n\mbox{.}$$
\end{lem}

The following result is due to Esser and Harary~\cite{Friedrich Frank}.
\begin{lem}[\cite{Friedrich Frank}]\label{Friedrich Frank}
For any $k$-partite graph $K=K_{n_1,\cdots,n_k}$ of order $n$, the characteristic polynomial $\Phi_K(\lambda)$ is given by $$\Phi_K(\lambda)=\lambda^{n-k}\left(\prod\limits_{i=1}^k(\lambda+n_i)-\sum\limits_{i=1}^k n_i\prod\limits_{j=1,j\neq i}^k(\lambda+n_j)\right).$$ And the spectral radius of $K$ is the largest root of $1-\sum\limits_{i=1}^k\frac{n_i}{\lambda+n_i}=0$.
\end{lem}

The following lemma shows that for a complete multipartite graph the more balanced the graph is, the larger will the spectral radius be.

\begin{lem}\label{comp partite}
For any $k$-partite graph $K_{n_1,\cdots,n_k}$ of order $n$, if there exist $i$ and $j$ with $n_i-n_j\geq 2$, then $\lambda(K_{n_1,\cdots,n_i-1,\cdots,n_j+1,\cdots,n_k})>\lambda(K_{n_1,\cdots,n_i,\cdots,n_j,\cdots,n_k})$.
\end{lem}

\begin{proof}
Let $A$ and $\tilde{A}$ be adjacent matrices of $K=K_{n_1,\cdots,n_k}$ and  $\Tilde{K}=K_{n_1,\cdots,n_i-1,\cdots,n_j+1,\cdots,n_k}$, respectively, where $\Tilde{K}$ is obtained from $K$ by moving a vertex $v$ in the $i$-th part $V_i$ to the $j$-th part $V_j$. 
Let $\lambda=\lambda(K)$ and $\Tilde{\lambda}=\lambda(\Tilde{K})$. 
%are the spectral radius and of $A$ and $\tilde{A}$ respectively, 
Let $\textbf{x}$ be a unit Perron vector of $A$. 
Note that all vertices in the same part of $K$ or $\Tilde{K}$ have the same corresponding components in its unit Perron vector. Denote the components corresponding to the vertices in the $\ell$-th part in $\bf{x}$ by $x_\ell$. 
Let $f(x)=\sum\limits_{\ell=1}^k\frac{n_\ell}{x+n_\ell}-1$. By Lemma~\ref{Friedrich Frank}, $\lambda$ is the largest root of $f(x)$. Clearly, $f(+\infty)=-1<0$ and $f(n_j)=\sum\limits_{\ell\neq i,j}^k\frac{n_\ell}{n_j+n_\ell}+\frac{n_i}{n_j+n_i}-\frac{1}{2}>0$. Hence, $\lambda>n_j$. Since $A\textbf{x}=\lambda \textbf{x}$, we have $\lambda x_m=\sum\limits_{\ell=1}^k n_\ell x_\ell-n_m x_m$, i.e.  $x_m=\sum\limits_{\ell=1}^k n_\ell x_\ell/(\lambda+n_m)$ for $m\in [k]$. Therefore,
\begin{eqnarray*}
{\bf{x}}^T(\Tilde{A}-A){\bf x}&=&\sum_{u\in V_i\setminus\{v\}}2x_ux_v-\sum\limits_{u\in V_j}2x_ux_v\\
                              &=&2(n_i-1)x_i^2-2n_jx_jx_i\\
                              &=&2x_i[(n_i-1)x_i-n_jx_j]\\
                              &=&2x_i\sum\limits_{\ell=1}^k n_\ell x_\ell\left( \frac{n_i-1}{\lambda+n_i}-\frac{n_j}{\lambda+n_j}\right)\\
                              &=&2x_i\sum\limits_{\ell=1}^k n_\ell x_\ell\frac{\lambda n_i-\lambda-n_j-\lambda n_j}{(\lambda+n_i)(\lambda+n_j)}\\
                              &>&0,
\end{eqnarray*}
the last inequality holds because $n_i\geq n_j+2$ and $\lambda>n_j$.
%, $\lambda n_i-\lambda-n_j-\lambda n_j>0$. 
Therefore, we have $\Tilde{\lambda}>\lambda$.

\end{proof}

Let $M$ be an $n\times n$ real symmetric matrix with the following block form
%and let $\pi$ be a $k$-partition of $[n]$, then $M$ can be write in the following block form
$$M=\left(\begin{array}{ccc}
M_{11} & \cdots & M_{1 k} \\
\vdots & \ddots & \vdots \\
M_{k 1} & \cdots & M_{k k}
\end{array}\right).$$
For $1\leq i,j\leq k$, let $b_{ij}$ denote the average row sum of $M_{ij}$. The matrix $B=(b_{ij})$ is called the {\em quotient matrix} of $M$. Moreover, if for each pair $i,j$, $M_{ij}$ has a constant row sum, then $B$ is called the {\em equitable quotient matrix} of $M$.

\begin{lem}[\cite{quotmatrix}]\label{quotmatrix}
Let $M$ be an $n\times n$ real symmetric matrix and let $B$ be an equitable quotient matrix of $M$. If $M$ is nonnegative and irreducible, then $\lambda(M)=\lambda(B)$, where $\lambda(M)~and~\lambda(B)$ are the largest eigenvalues of $M$ and $B$, respectively.
\end{lem}

For two non-adjacent vertices $u,v$ in a graph $G$, we  define the {\em switching operation $u\rightarrow v$} as deleting the edges joining $u$ to its neighbors  and adding new edges connecting $u$ to the neighborhood of $v$. Let  $G_{u\rightarrow v}$ be the graph
obtained from $G$ by the switching operation $u\rightarrow v$, that is $V(G_{u\rightarrow v})=V(G)$ and
$$E(G_{u\rightarrow v})=\left(E(G)\setminus E_G(u, N_G(u))\right)\cup E_G(u, N_G(v)),$$
where $E_G(S, T)$ is the set of edges in $G$ with one end in $S$ and the other in $T$  for disjoint subsets $S, T\subset V(G)$.
Note that the edges between  $u$ and the common neighbors of $u$ and $v$ remain unchanged by the definition of $G_{u\rightarrow v}$.
For two disjoint independent sets $S$ and $T$ in a graph $G$, if all vertices in $S$ (resp. $T$) have the same neighborhood $N_G(S)$ (resp. $N_G(T)$), we similarly define $G_{S\rightarrow T}$ to be the graph obtained from $G$ by deleting the edges between $S$  and $N_G(S)$ and adding new edges connecting $S$ and $N_G(T)$. 
\begin{prop}\label{PROP:p1}
	%(1) For $r\ge 2 $ and two non-adjacent vertices $u,v$ in a graph $G$, either $G'=G_{u\rightarrow v}$ or $G'=G_{v\rightarrow u}$ has the property that $N(G', K_r)\ge N(G, K_r)$, the equality holds if and only if $d_G^{(r-1)}(u)=d_G^{(r-1)}(v)$.
	
	For $r\ge 2$ and two disjoint independent sets $S$ and $T$ in a graph $G$, if all of vertices in $S$ (resp. $T$) have the same neighborhood $N_G(S)$ (resp. $N_G(T)$) and $E_G(S, T)=\emptyset$, then either $G'=G_{S\rightarrow T}$ or $G'=G_{T\rightarrow S}$ has the property that  $\lambda(G^{\prime}) \geq$ $\lambda(G)$.
	% the equality holds if and only if $d_G^{(r-1)}(S)=d_G^{(r-1)}(T)$. 
\end{prop}
\begin{proof}
	Let $S$ and $T$ be two such independent sets of $G$. Let  $\textbf{x}$ be a unit Perron vector of $A(G)$.  Without loss of generality, suppose $\sum\limits_{z\in N_{G}(T)}x_z\geq \sum\limits_{z\in N_{G}(S)}x_z$. Let $G^{\prime}=G_{S \rightarrow T}$. Then 
	\begin{eqnarray*}
	{\bf{x}}^T(A(G^{\prime})-A(G)){\bf x}&=&\sum\limits_{u\in S}\sum\limits_{z\in N_{G}(T)}2 x_u x_z -  \sum\limits_{u\in S}\sum\limits_{z\in N_{G}(S)}2 x_u x_z\\
	                                 &=& 2 \sum\limits_{u\in S}x_u \left(\sum\limits_{z\in N_{G}(u)}x_z -  \sum\limits_{z\in N_{G}(v)}x_z\right)\geq 0.
	\end{eqnarray*}
Therefore, we have $\lambda(G^{\prime}) \geq$ $\lambda(G)$.
\end{proof}

\section{Proof of Theorem~\ref{main}}
Now we are ready to give the proof of the main theorem.
\begin{proof}[Proof of Theorem~\ref{main}]
Suppose $n\geq 4s^2+9s$. Let $G$ be an  $n$-vertex graph with maximum spectral radius over all $\{K_{k+1},M_{s+1}\}$-free graphs. 
%Let $A(G)$ be adjacent matrix of $G$ and   $\lambda$ the spectral radius of $G$. 
Let $\lambda=\lambda(G)$ and $\textbf{x}$ be a unit Perron vector of $A(G)$. 
We show that $\lambda(G)\le\lambda(G_k(n,s))$. %{\color{blue} and the equality holds if and only if $G\cong G_k(n,s)$.}

Since $G$ is $M_{s+1}$-free, by Lemma~\ref{Tutte-Berge}, there is a vertex set $B\subset V(G)$ such that  $G-B$ consists of odd components $G_1,\ldots, G_m$, and 
\begin{equation}\label{EQ: e1}
|B|+\sum_{i=1}^{m}\frac{|V(G_i)|-1}{2}=s\mbox{.}
\end{equation}
%Now we pick $B,A_1,A_2,\cdots, A_m$ of $G$ as in the Lemma~\ref{Tutte-Berge}. 
Let $A_i=V(G_i)$ and $|A_i|=a_i$ for $i\in[m]$. Denote $A=\cup_{i=1}^mA_i$. 
Let $I_G(A)=\{ i\in[m] : a_i=1\}$. We may choose $G$ maximizing $|I_G(A)|$ (assumption (*)). 
Let $|B|=b$. 
%Let $B,A_i,i=1\cdots m$ be like in the Lemma~\ref{Tutte-Berge}, $A=G-B=\bigcup A_i$. 
Then we have $b\leq s$ and $a_i\leq 2s+1$.
%Next we will prove the following lemma by Zykov symmetrization method introduced in \cite{zykov}.

Define two vertices $u$ and $v$ in  $B$ are equivalent if and only if $N_G(u)=N_G(v)$. Clearly, it is an equivalent relation. Therefore, the vertices of $B$ can be partitioned into equivalent classes according to the equivalent relation defined above.  We may choose $G$ (among graphs $G$ satisfying assumption (*)) with the minimum number of equivalent classes of $B$ (assumption (**)).  
Note that each equivalent class of $B$ is an independent set of $G$ by the definition of the equivalent relation. 
We first claim that every two non-adjacent vertices of $B$ have the same neighborhood  (a spectral version of Lemma 2.1 in~\cite{AF}),
%,  which is also a simple consequence of the Zykov symmetrization method introduced in~\cite{Zykov}, 
for completeness we include the proof.

\begin{claim}\label{2.1}
Every two non-adjacent vertices of $B$ have the same neighborhood.
\end{claim}
\begin{proof}
Suppose there are two non-adjacent vertices $u, w\in B$ with $N_{G}(u)\neq N_{G}(w)$. Then $u$ and $w$ must be in different equivalent classes $U$ and $W$ by the definition of the equivalence. Since $uw\notin E(G)$, we have $E_G(U, W)=\emptyset$. Without loss of generality,  suppose $\sum\limits_{z\in N_{G}(w)}x_z\geq \sum\limits_{z\in N_{G}(u)}x_z$. 
Let $G'=G_{U\rightarrow W}$. By Proposition~\ref{PROP:p1}, $\lambda(G')\ge \lambda(G)$. Now we show that $G'$ is $\{K_{k+1}, M_{s+1}\}$-free too. Clearly,  $G'-B$ still consists of odd components $G_1,\ldots, G_m$. 
Hence $G'$ is $M_{s+1}$-free by  Lemma~\ref{Tutte-Berge}. 
If $G'$ contains a copy $T$ of $K_{k+1}$, we must have a vertex $u'\in V(T)\cap U$. 
%Choose a vertex $w'\in W$. 
Since $N_{G'}(u')=N_{G'}(w)=N_G(w)$,  $(V(T)\setminus\{u'\})\cup\{w\}$ induces a copy of $K_{k+1}$ in $G$, a contradiction. Hence, $G_{U\rightarrow W}$ is $\{K_{k+1}, M_{s+1}\}$-free. 
By the extremality of $G$, we have $\lambda(G')= \lambda(G)$. 
But the number of equivalent classes of $G'$ ($U$ and $W$  merge into one class in $G'$) is less than the one in $G$, a contradiction to the assumption (**).
%If $\sum\limits_{z\in N_{G}(v)}x_z >  \sum\limits_{z\in N_{G}(u)}x_z$, then replacing the neighbourhood of $u$ by that of $v$. Denote the new graph by $\tilde{G}$, we have $$x^T(A(\tilde{G})-A(G))x=\sum\limits_{z\in N_{G}(v)}2 x_u x_z -  \sum\limits_{z\in N_{G}(u)}2 x_u x_z \geq 2 x_u(\sum\limits_{z\in N_{G}(u)}x_z -  \sum\limits_{z\in N_{G}(v)}x_z)>0$$
%So $\lambda(\tilde{G})>\lambda(G)$. $\tilde{G}$ is still $K_{k+1}$-free, since any new clique $K$ of size $k+1$ must contain $u$, and it cannot contain $v$, but $(K-\{u\})\cup\{v\}$ is a clique of size $k+1$, contradiction! $\tilde{G}$ is also $M_{s+1}$-free by Lemma~\ref{Tutte-Berge}. Hence, $\tilde{G}$ is a $\{K_{k+1},M_{s+1}\}$-free graph with larger spectral radius, contradiction!

%If $\sum\limits_{z\in N_{G}(v)}x_z =  \sum\limits_{z\in N_{G}(u)}x_z$, then we still replace the neighbourhood of $u$ by that of $v$ and denote the new graph by $\tilde{G}$. Similarly, $\tilde{G}$ is a $\{K_{k+1},M_{s+1}\}$-free graph with $\lambda(\tilde{G})>\lambda(G)$. %Moreover, since  $d_{G}(v)\le d_{G}(u)$, we have $|G'|\le |G|$. If $|G'|<|G|$, then we conclude with a contradiction. Hence, $d_{G}(v)= d_{G}(u)$ and $|G'|=|G|$. 
%Repeatedly doing the same operation with proper order (lexicographical order for example), we can make every two non-adjacent vertices of $B$ have the same neighborhood in $G$.
\end{proof}

%By Claim~\ref{2.1} and $G$ is $K_{k+1}$-free, $G[B]$ is a complete multi-partite graph, and its number of partitions is smaller than $k$. Let $B_1, B_2,\cdots, B_k$ be its partition sets, $|B_i|=b_i\ge 0, i=1,\cdots,k$. 

By Claim~\ref{2.1} and $G$ is $K_{k+1}$-free, $G[B]$ is a complete $\ell$-partite graph with $\ell\le k$. Let its partition sets be $B_1, \ldots, B_\ell$ and let $B_{\ell+1}=\cdots=B_k=\emptyset$ if $\ell<k$. Let $b_i=|B_i|$ for $i\in[k]$. 
Without loss of generality, 
%assume $b_1\ge b_2\ge\ldots\ge b_k\ge 0$. We can 
assume that $\sum\limits_{v\in B_1}x_v\geq\cdots\geq\sum\limits_{v\in B_k}x_v$. 
By Claim~\ref{2.1}, if there is a vertex in $B_i$ adjacent to $v\in A_j$ then $B_i\subseteq N_G(v)$. 

%Let $\alpha_i=\max\{x_v:v\in A_i\}$, wlog, suppose $\alpha_1\geq\cdots\geq\alpha_m$.

\begin{claim}\label{2.2}
$a_2=a_3=\cdots=a_m=1$.
\end{claim}
\begin{proof}
Suppose $v_1$ is a vertex in $A$ with $\sum\limits_{u\in N_G(v_1)}x_u=\max\limits_{v\in A}\sum\limits_{u\in N_G(v)}x_u$. Without loss of generality, suppose $v_1\in A_1$. We prove by contradiction. Suppose there is an $a_i$ with $a_i\not=1$ for some $2\le i\le m$.

If $|G[A_1]|=0$, let $G'$ be the resulting graph by applying the switching operations $u\rightarrow v_1$ for all vertices $u\in A\setminus\{v_1\}$ one by one. Then we have $|G'[A]|=0$. By Proposition~\ref{PROP:p1}, $\lambda(G')\ge\lambda(G)$. With the same discussion as in the proof of Claim~\ref{2.1}, we have that $G'$ is still $\{K_{k+1},M_{s+1}\}$-free. But $|I_{G'}(A)|=m>|I_G(A)|$, a contradiction to the assumption (*).
% i.e. $a_1=1$, then we replace the neighborhoods of all the vertices in $A\backslash\{v_1\}$ by $N_{G}(v_1)$. Denote the new graph by $\tilde{G}$, $|E_{\tilde{G}}(A)|=0$ i.e. $a_2=a_3=\cdots=a_m=1$. And similarly as the proof in Lemma~\ref{2.1}, we have $\lambda(\tilde{G})\geq\lambda(G)$ and $\tilde{G}$ is still $\{K_{k+1},M_{s+1}\}$-free.

If $|G[A_1]|>0$, i.e. $a_1\geq 3$. 
%If $a_2=\ldots=a_m=1$, we are done. Now, 
Without loss of generality, assume $a_2\ge 3$. 
Since $G[A_2]$ is connected, we can pick two vertices, say $u_1,u_2$ in $A_2$ such that $G[A_2\backslash\{u_1,u_2\}]$ is still connected ($u_1, u_2$ exist, for example, we can pick two leaves of a spanning tree of $G[A_2]$). 
Let $G_1$ be the resulting graph by applying the switching operations $u_1\rightarrow v_1$ and $u_2\rightarrow v_1$ one by one. 
With similar discussion as in the above case, we have $\lambda(G_1)\ge\lambda(G)$ and $G_1$ is $\{K_{k+1},M_{s+1}\}$-free. 
Continue the process after $t=\frac{a_2-1}{2}$ steps, we obtain a graph $G_t$ with $\lambda(G_t)\ge\lambda(G)$ and $G_t$ is $\{K_{k+1},M_{s+1}\}$-free. But $|I_{G_t}(A)|=|I_G(A)|+1$, a contradiction to the assumption (*). 
%We can pick two leaves $u_1,u_2$ of the spanning tree of $G[A_2]$ with $G[A_2\backslash\{u_1,u_2\}]$ is still connected and odd. Now replace the neighborhoods of $u_1,u_2$ by $N_{G}(v_1)$, and denote the new graph by $G'$, $a_1 '=a_1+2$ and $a_2'=a_2-2$. Similarly as the proof in Lemma~\ref{2.1}, $\lambda(G')\geq\lambda(G)$ and $G'$ is still $\{K_{k+1},M_{s+1}\}$-free. Repeat the above operation until $a_i= 1,i=2,\cdots,m$. Denote the new graph by $\tilde{G}$, it is easy to see that $\lambda(\tilde{G})\geq\lambda(G)$ and $\tilde{G}$ is $\{K_{k+1},M_{s+1}\}$-free, completing the proof.
\end{proof}

The following proof is divided into two cases according to $a_1$.
\begin{case}
$a_1=1$. 
\end{case}
In this case $b=s$ by (\ref{EQ: e1}). By Claim~\ref{2.2}, $A$ is an independent set of $G$. Let $A=\{v_1,\ldots,v_m\}$. 

If $b_k= 0$, then $G[B]$ is a complete $\ell$-partite graph on $s$ vertices with $\ell\le k-1$. We may assume $\ell=k-1$ (Otherwise, we can add new edges in $G[B]$ to make it $(k-1)$-partite and this operation will increase the spectral radius of $G$ by the Perron-Frobenius Theorem, a contradiction to the maximality of $G$). With the same reason, we can add all missing edges between sets $A$ and $B$ to make $G$ a complete $k$-partite graph.  
%Otherwise, first assume that the number of partitions in $B$ is smaller than $k-1$, we can pick any partition and divide it into two parts, then add all edges between these two new parts. Next we assume that there exists two vertices $v\in A$ and $u\in B-B_k$ such that $uv\notin E(G)$, we can add a new edge $uv$, These operations will increase the spectral radius of $G$ but the new graph is still $\{K_{k+1},M_{s+1}\}$-free, contradiction! 
%Now by Lemma~\ref{comp partite} and $b=s$, it is easy to see that $\lambda(G)$ will reach the maximum if and only if $G\cong G_k(n,s)$.
Now by Lemma~\ref{comp partite}, we have $\lambda(G)\le \lambda(G_k(n,s))$, and the equality holds if and only if $G\cong G_k(n,s)$.

If $b_k\neq 0$, then $G[B]$ is a complete $k$-partite graph on $s$ vertices.
Since $G$ is $K_{k+1}$-free, each vertex in $A$ is only adjacent to $k-1$ parts in $B$.
%i.e. $B$ has $k$ non-empty partitions, since $G$ is $\{K_{k+1},M_{s+1}\}$-free, each vertex in $A$ is only adjacent to $k-1$ partitions in $B$. 
By the assumption $\sum\limits_{v\in B_1}x_v\geq\cdots\geq\sum\limits_{v\in B_k}x_v$ and the maximality of $G$, we may assume every vertex of $A$ is adjacent to $B_1,\cdots,B_{k-1}$ (the only possible exception is when $\sum\limits_{v\in B_1}x_v=\cdots=\sum\limits_{v\in B_k}x_v$, in this case, we can relabel $B_1, \ldots, B_k$ and do switching operations in vertices of $A$ to  obtain a new graph with the non-decrease spectral radius and the desired property).
%the vertices in $A$ if $B_k\in N_G(v_i)$ and $B_j\notin N_G(v_i)(j\neq k)$, then we replace $B_k$ by $B_j$. Since $\sum\limits_{v\in B_j}x_v\geq\sum\limits_{v\in B_k}x_v$, the spectral radius of the new graph won't decrease.) 
Now combine $B_k$ and $A$ as one part, we  obtain that $G$ is a complete $k$-partite graph. Since $\sum\limits_{i=1}^{k-1}b_i\leq s-1$, by Lemma~\ref{comp partite}, we have $\lambda(G)<\lambda(G_k(n,s))$. This completes the proof of the case.

\begin{case}
    $a_1\geq 3$
\end{case}
In this case $b+\frac{a_1-1}2=s$. Since $G$ is $K_{k+1}$-free and has maximum spectral radius, we also can assume that $G[B]$ is a complete $\ell$-partite graph with $\ell=k-1$ or $k$ and each vertex in $A$ is only adjacent to the first $k-1$ parts in $G[B]$. Now let  $\tilde{A}=B_k\cup(A\setminus A_1)$ and $a=|\tilde{A}|$. Then $\tilde{A}$ is an independent set of $G$. 
To finish the proof, we will show that $\lambda(G)<\lambda(G_k(n,s))$. 
To do this, let $\tilde{G}$ be the graph obtained by adding all missing edges (if any) between the sets $\tilde{A}$ and $B\setminus B_{k}$, all missing edges (if any) between $A_1$ and $B\setminus B_{k}$, and all missing edges (if any ) in $A_1$, i.e. $\tilde{G}=G[B\setminus  B_k]\vee(\overline{K_{a}}\cup K_{a_1})$. Clearly, $G\subseteq \tilde{G}$. Hence we have $\lambda(\tilde{G})\ge\lambda(G)$. Therefore, it is sufficient to show that $\lambda(G_k(n,s))>\lambda(\tilde{G})$.

\begin{claim}\label{A}
$\lambda(G_k(n,s))>\lambda(\tilde{G})$.
\end{claim}
\begin{proof}
%Apartite the vertex set of $\tilde{G}$ by $\pi : V(B_1)\cup V(B_2)\cup\cdots\cup v(B_k-1)\cup V(\tilde{A})\cup V(A_1)$. 
The quotient matrix of $A(\tilde{G})$ according to the partition $B_1\cup \cdots\cup B_{k-1}\cup \tilde{A}\cup A_1$ is
%插入矩阵M
$$M=\begin{pmatrix}
0      & b_2 &\cdots & b_{k-1} & a      & a_1      \\
b_1    & 0   &\cdots & b_{k-1} & a      & a_1     \\
\vdots & \vdots &\ddots & \vdots  & \vdots & \vdots  \\
b_1    & b_2 &\cdots & 0       & a      & a_1     \\
b_1    & b_2 &\cdots & b_{k-1} & 0      & 0     \\
b_1    & b_2 &\cdots & b_{k-1} & 0      & a_1-1     
\end{pmatrix}$$
By Lemma~\ref{quotmatrix}, we have $\lambda(\tilde{G})=\lambda(M)$, where $\lambda(M)$ is the largest eigenvalue of $M$. It can be calculated that the characteristic polynomial of $M$ is
\begin{eqnarray*}\label{EQ: e2}
	\Phi_M(\lambda)
	%&=&(a_1-1-\lambda)(\lambda+a)\prod\limits_{i=1}^{k-1}(-\lambda-b_i)\left(-\frac{\lambda}{a+\lambda}+\sum\limits_{i=1}^{k-1}\frac{b_i}{b_i+\lambda}\right)-\lambda a_1\prod\limits_{i=1}^{k-1}(-\lambda-b_i)\sum\limits_{i=1}^{k-1}\frac{b_i}{b_i+\lambda}\nonumber%\\
    =(\lambda^2+(a+1)\lambda+a(1-a_1))\prod\limits_{i=1}^{k-1}(-\lambda-b_i)\left(-\frac{\lambda(\lambda+1-a_1)}{\lambda^2+\lambda+a(\lambda+1-a_1)}+\sum\limits_{i=1}^{k-1}\frac{b_i}{b_i+\lambda}\right).%\\
	%&=&-\frac{\lambda(\lambda+1-a_1)}{\lambda^2+\lambda+a(\lambda+1-a_1)}+\sum\limits_{i=1}^{k-1}\frac{b_i}{b_i+\lambda}.
\end{eqnarray*}
%$$\Phi_M(\lambda)=(a_1-1-\lambda)(\lambda+a)\prod\limits_{i=1}^{k-1}(-\lambda-b_i)\left(-\frac{\lambda}{a+\lambda}+\sum\limits_{i=1}^{k-1}\frac{b_i}{b_i+\lambda}\right)-\lambda a_1\prod\limits_{i=1}^{k-1}(-\lambda-b_i)\sum\limits_{i=1}^{k-1}\frac{b_i}{b_i+\lambda}.$$
Thus, $\lambda(M)$ is the largest root of $\Phi_M(\lambda)$.
%\begin{equation}\label{EQ: e2}
%	-\frac{\lambda(\lambda+1-a_1)}{\lambda^2+\lambda+a(\lambda+1-a_1)}+\sum\limits_{i=1}^{k-1}\frac{b_i}{b_i+\lambda}=0.
%\end{equation}
{Let %$$f_{0}(\lambda)=\Phi_M(\lambda)$$.
$$f_{0}(\lambda)=-\frac{\lambda(\lambda+1-a_1)}{\lambda^2+\lambda+a(\lambda+1-a_1)}+\sum\limits_{i=1}^{k-1}\frac{b_i}{b_i+\lambda}$$ 
and $$h(\lambda)=\lambda^2+(a+1)\lambda+a(1-a_1).$$}
Since $a_1\leq 2s+1$, $1\le b_i\leq b\leq s$ and $n\geq 4s^2+9s$, we have $a=n-a_1-b\geq 4s^2+6s-1$. 
Then we have $$f_0(a_1)=-\frac{a_1}{a_1^2+a_1+a}+\sum\limits_{i=1}^{k-1}\frac{b_i}{b_i+a_1}\geq\frac{k-1}{2s+2}-\frac{1}{2\sqrt{a}+1}\geq\frac{1}{2s+2}-\frac{1}{4s-1}>0$$
and
$$f_0(+\infty)=\lim_{\lambda\to+\infty}\left(-1+\frac{(a+a_1)\lambda-a(a_1-1)}{\lambda^2+\lambda+a(\lambda+1)-aa_1}+\sum\limits_{i=1}^{k-1}\frac{b_i}{b_i+\lambda}\right)=-1<0.$$
{Thus the largest root of $f_0(\lambda)$ is larger than  $a_1$. 
Since $h(a_1)=a_1^2+a_1+a>0$ and $a+1>0$, the largest root of $f_0(\lambda)$ and  $\Phi_M(\lambda)$ are the same.} Therefore, we have $\lambda(M)>a_1$.

Next, we will prove that $\lambda(M)<\lambda(G_k(n,s))$ by shifting vertices from $A_1$ to $\tilde{A}$ and some $B_i$ for $i\in[k-1]$.
Specifically, arbitrarily choose an $i\in [k-1]$, let $\tilde{G}_1$ be the graph obtained from $\tilde{G}$ by shifting one vertex from $A_1$ to  $\tilde{A}$ and one vertex  from $A_1$ to some $B_i$, where when we shift a vertex from a set $X$ to another set $Y$, we delete the edges between the vertex and its neighbors and adding new edges connecting it to the neighborhood of $Y$. Note that $\tilde{G}=K_{b_1,\ldots, b_{k-1}}\vee(\overline{K_{a}}\cup K_{a_1})$. Then  $\tilde{G}_1=K_{b_1,\ldots, b_i+1,\ldots,b_{k-1}}\vee(\overline{K_{a+1}}\cup K_{a_1-2})$.
% $E_{\tilde{G}}(w, N_{\tilde{G}}(w))$ and add the edges between $B$ and its neighbors. 
{Let 
\begin{eqnarray*}
f_1(\lambda)
%&=&-\frac{\lambda(\lambda+1-a_1')}{\lambda^2+\lambda+a'(\lambda+1-a_1')}+\sum\limits_{i=1}^{k-1}\frac{b_i'}{b_i'+\lambda}\\
&=&-\frac{\lambda(\lambda+3-a_1)}{\lambda^2+\lambda+(a+1)(\lambda+3-a_1)}+\frac{b_i+1}{b_i+1+\lambda}+\sum\limits_{l=1,l\neq i}^{k-1}\frac{b_l}{b_l+\lambda}.
\end{eqnarray*}}
Then 
\begin{eqnarray*}
f_1(\lambda)-f_0(\lambda)&=&\frac{\lambda}{\lambda^2+(2b_i+1)\lambda+b_i^2+b_i}\\
                    &&-\frac{\lambda(\lambda^2+(2a_1-2)\lambda-a_1^2+4a_1-3)}{(\lambda^2+(a+1)\lambda+a(1-a_1))(\lambda^2+(a+2)\lambda+(a+1)(3-a_1))}\\
                    &:=& \frac{\lambda}{g_1(\lambda)}-\frac{\lambda g_2(\lambda)}{g_3(\lambda)g_4(\lambda)}\\
                    &=&\frac{\lambda g_3(\lambda) g_4(\lambda)-\lambda g_1(\lambda) g_2(\lambda)}{g_1(\lambda)g_3(\lambda)g_4(\lambda)}.
\end{eqnarray*}
Since $a_1\leq 2s+1$, $1\le b_i\le s$, $a\geq 4s^2+6s-1$,  and  $\lambda(M)> a_1$, we have 
\begin{eqnarray*}
g_3(\lambda(M))-g_2(\lambda(M))
%&=&(\lambda^2+(2a_1-2)\lambda-a_1^2+4a_1-3)-(\lambda^2+\lambda+a(\lambda+1)-aa_1)\\
&=&(1+a-2a_1+2)\lambda(M)+a(1-a_1)+(a_1-1)(a_1-3)\\
&>&(1+a-2a_1+2)a_1+a(1-a_1)+(a_1-1)(a_1-3)\\
&=&a-a_1^2-a_1+3\\
&\ge &4s^2+6s-1-(2s+1)^2-(2s+1)+3\\
&\geq&0,
\end{eqnarray*}
%It implies that
%\begin{eqnarray*}
%f(\lambda(M))'-f(\lambda(M))&>&\frac{\lambda(M)}{\lambda(M)^2+(2b_i+1)\lambda(M)+b_i^2+b_i}\\&&-\frac{\lambda(M)}{\lambda(M)^2+\lambda(M)+(a+1)(\lambda(M)+1)-(a+1)(a_1-2)}
%\end{eqnarray*}
%Since 
and
\begin{eqnarray*}
g_4(\lambda(M))-g_1(\lambda(M))&=&(a-2b_i+1)\lambda(M)+(a+1)(3-a_1)-b_i^2-b_i\\
%(\lambda(M)^2+\lambda(M)+(a+1)(\lambda(M)+1)-(a+1)(a_1-2))-(\lambda(M)^2+(2b_i+1)\lambda(M)+b_i^2+b_i)\\
&>&(a-2b_i+1)a_1+(a+1)(3-a_1)-b_i^2-b_i\\
&\geq&3(4s^2+6s)-5s^2-3s\\
&>&0.
\end{eqnarray*}
Therefore, we have $f_1(\lambda(M))-f_0(\lambda(M))>0$, which implies that the spectral radius increases after one shifting operation.
%when we doing the shifting, the largest root of $f(\lambda)=0$ gets larger. 
%Actually, the calculations above are all right when $a_1=1$, which means that the spectral radius of $\tilde{G}$ is smaller than the largest root of $$-\frac{\lambda^2}{\lambda^2+\lambda+a\lambda}+\sum\limits_{i=1}^{k-1}\frac{b_i}{b_i+\lambda}=0$$
%which is exactly the spectral radius of 
Therefore, after $t=\frac{a_1-1}2$ times of shifting operations, we get a complete $k$-partite graph $\tilde{G}_t=K_{b'_1,\ldots,b'_{k-1},a'}$, where $\sum\limits_{i=1}^{k-1} b'_i=s-b_k\leq s$ and $a'=n-s+b_k\ge n-s$. By Lemma~\ref{comp partite}, we have $\lambda(G_k(n,s))\ge\lambda(\tilde{G}_t)>\lambda(\tilde{G})$. This completes the proof of Case 2.
\end{proof}

%Combining the conclusions of Case 1 and 2, we have $SPEX(n,\{K_{k+1},M_{s+1}\})=\{G_k(n,s)\}$, completing 
The proof of Theorem~\ref{main} is completed.

\end{proof}

\section{Concluding Remarks}
In this note, we determine $\spex(n,\{K_{k+1},M_{s+1}\})$ when $n>4s^2+9s$, we believe that the lower bound of $n$ can be optimized, and when $n$ is small, the extremal graph will be $T_k(n)$. We leave this as a problem.

\noindent{\bf Acknowledgment:} {The work was supported by the National Natural Science Foundation of China (No. 12071453), the National Key R and D Program of China(2020YFA0713100), the Anhui Initiative in Quantum Information Technologies (AHY150200)  and the Innovation Program for Quantum Science and Technology, China (2021ZD0302904).}

\noindent{\bf Data Availability:} Data sharing is not applicable to this article as no datasets
were generated or analyzed during the current study.

\end{document}